\documentclass[fleqn, 11pt]{amsart}
\usepackage{amsmath, amssymb, amsthm, graphicx}
\usepackage{url}
\usepackage{enumerate}

\newtheorem{theorem}{Theorem}[section]
\newtheorem{lemma}[theorem]{Lemma}

\theoremstyle{definition}
\newtheorem{definition}[theorem]{Definition}
\newtheorem{remark}[theorem]{Remark}
\numberwithin{equation}{section}

\author{Pieter Tibboel}
\address{
Department of Mathematical Sciences\\
Xi'an Jiaotong-Liverpool University\\
Suzhou, China}
\email{Pieter.Tibboel@xjtlu.edu.cn}


\title[Polygonal negative hyperbolic rotopulsators]{Polygonal negative hyperbolic rotopulsators of the curved $n$-body problem. }

\begin{document}
\begin{abstract}
  For the $n$-body problem in spaces of negative constant Gaussian curvature, we prove for a class of negative hyperbolic rotopulsators that if that class exists, the configurations of the point masses of these rotopulsators have to be regular polygons if the rotopulsators are not relative equilibria. Additionally, we prove that if the rotopulsators are relative equilibria, there exists at most one such solution.
\end{abstract}
\maketitle

\section{Introduction}
  By $n$-body problems we mean problems where we are to determine the dynamics of a number of $n$ point masses as dictated by a system of ordinary differential equations. The $n$-body problem in spaces of constant Gaussian curvature, or curved $n$-body problem for short, is described by a system of differential equations (see (\ref{EquationsOfMotion Curved})) that generalises the classical, or Newtonian $n$-body problem to spaces of constant Gaussian curvature. Solutions to an $n$-body problem where the point masses describe a configuration that maintains the same shape and size over time are called relative equilibria. A rotopulsator, or rotopulsating orbit is a solution of the curved $n$-body problem for which the shape of the configuration of the point masses stays the same over time, but the size may change. An important reason to study the curved $n$-body problem, and relative equilibria and rotopulsators in particular, is to identify orbits that are unique to a particular space (see \cite{DK}). For example: Diacu, P\'erez-Chavela and Santoprete (see \cite{DPS1}, \cite{DPS2}) showed that rotopulsators (called homographic orbits in those papers) that have an equilateral triangle configuration and unequal masses, only exist in spaces of zero curvature. As the Sun, Jupiter and the Trojan asteroids form the vertices of an equilateral triangle, the region between these three objects likely has zero curvature. Rotopulsators were first introduced in \cite{DK}, where it was proven that there are five different types of rotopulsators, two for the positive curvature case (spheres) and three for the negative curvature case (hyperboloids). For subclasses of these five different types it was proven in \cite{DT} for $n=4$ that if the rotopulsators have rectangular configurations, they have to be squares. For two of these five subclasses it was proven in \cite{T2} for general $n$ that if the configuration forms a polygon and the rotopulsator is not a relative equilibrium, that polygon has to be a regular polygon. A logical avenue of research then is to investigate to what extent we can generalise the remaining three results from \cite{DT}. In this paper we will generalise the results for two of these three remaining classes, the so-called negative hyperbolic polygonal rotopulsators and negative elliptic hyperbolic polygonal rotopulsators (see \cite{DK}, \cite{DT} and Definition~\ref{Definition negative hyperbolic}). Specifically, we will prove the following result:
  \begin{theorem}\label{Main Theorem 1}
    Let $q_{1},...,q_{n}$ be a negative hyperbolic, or negative elliptic hyperbolic polygonal rotopulsator, as in Definition~\ref{Definition negative hyperbolic}. If $\rho$ is not constant, then the $q_{i}$, $i\in\{1,...,n\}$ are the vertices of a regular polygon. If $\rho$ is constant, then the $q_{i}$, $i\in\{1,...,n\}$ are the vertices of a unique polygonal relative equilibrium.
  \end{theorem}
  \begin{remark}
    Because of the proof of Theorem~\ref{Main Theorem 1}, the rotopulsators in this paper are in fact the rotopulsators that were investigated in \cite{T2}. It was proven in \cite{T2} and in \cite{D2} (in the latter paper they were called 'homographic orbits' instead of 'rotopulsators') that these rotopulsators exist.
  \end{remark}
  \begin{remark}
    In \cite{DT}, in Theorem~6 and Theorem~7, it was stated that for \\$n=4$ rectangular negative hyperbolic rotopulsators and rectangular negative elliptic hyperbolic rotopulsators do not exist. These statements are not in conflict with Theorem~\ref{Main Theorem 1}, as the third and fourth coordinates of the point masses of the negative hyperbolic rotopulsators in \cite{DT} are constructed to be coordinates of distinct points on a hyperbola, while we do not impose that restriction in this paper.
  \end{remark}
  The curved $n$-body problem goes back as far as the 1830s with Bolyai and Lobachevsky, who independently proposed a curved 2-body problem in hyperbolic space (see \cite{BB} and \cite{Lo}). While of significant interest to great mathematicians such as Dirichlet, Schering (see \cite{S1}, \cite{S2}), Killing (see \cite{K1}, \cite{K2}, \cite{K3}), Liebmann (see \cite{L1}, \cite{L2}, \cite{L3}), Kozlov, Harin (see \cite{KH}), Cari\~nena, Ra\~nada and Santander (see \cite{CRS}), a working model for the $n\geq 2$ case was not found until 2008 by Diacu, P\'erez-Chavela and Santoprete (see \cite{DPS1}, \cite{DPS2} and \cite{DPS3}). This breakthrough then gave rise to further results for the $n\geq 2$ case in \cite{D1}--\cite{D5}, \cite{D7}, \cite{DK}, \cite{DPo}, \cite{DT}, \cite{T}--\cite{T4} and \cite{Z1}.

  The remainder of this paper is constructed as follows: In section~\ref{Background theory} we will discuss needed background theory, after which we will prove Theorem~\ref{Main Theorem 1} in section~\ref{Section proof of main theorem 1}.

%
  \section{Background theory}\label{Background theory}
  \begin{definition}
    Let $\sigma=\pm 1$. The $n$-body problem in spaces of constant Gaussian curvature is the problem of finding the dynamics of point masses \begin{align*}q_{1},...,\textrm{ }q_{n}\in\mathbb{M}_{\sigma}^{3}=\{(x_{1},x_{2},x_{3},x_{4})\in\mathbb{R}^{4}|x_{1}^{2}+x_{2}^{2}+x_{3}^{2}+\sigma x_{4}^{2}=\sigma\},\end{align*} with respective masses $m_{1}>0$,..., $m_{n}>0$, determined by the system of differential equations
  \begin{align}\label{EquationsOfMotion Curved}
   \ddot{q}_{i}=\sum\limits_{j=1,\textrm{ }j\neq i}^{n}\frac{m_{j}(q_{j}-\sigma(q_{i}\odot q_{j})q_{i})}{(\sigma -\sigma(q_{i}\odot q_{j})^{2})^{\frac{3}{2}}}-\sigma(\dot{q}_{i}\odot\dot{q}_{i})q_{i},\textrm{ }i\in\{1,...,\textrm{ }n\},
  \end{align}
  where for $x$, $y\in\mathbb{M}_{\sigma}^{3}$  the product $\cdot\odot\cdot$ is defined as
  \begin{align*}
    x\odot y=x_{1}y_{1}+x_{2}y_{2}+x_{3}y_{3}+\sigma x_{4}y_{4}.
  \end{align*}
  \end{definition}
  Next, we will define negative hyperbolic polygonal rotopulsators and negative elliptic-hyperbolic polygonal rotopulsators:
  Let \begin{align*}
    T(x)=\begin{pmatrix}
    \cos{x} & -\sin{x} \\
    \sin{x} & \cos{x}
  \end{pmatrix}\textrm{ and }S(x)=\begin{pmatrix}
    \cosh{x} & \sinh{x} \\
    \sinh{x} & \cosh{x}
  \end{pmatrix}\end{align*} be $2\times 2$ matrices. Then
  \begin{definition}\label{Definition negative hyperbolic}
    If $\sigma=-1$ and there exist scalar, twice differentiable functions $x_{i}$, $y_{i}$, $i\in\{1,...,n\}$, $\phi$, $\rho\geq 0$, for which $x_{i}^{2}+y_{i}^{2}-\rho^{2}=-1$ and there exist constants  $\beta_{1},...,\beta_{n}\in\mathbb{R}$, such that for a solution $q_{1}$,...,$q_{n}$ of (\ref{EquationsOfMotion Curved}) we have that
    \begin{align}\label{q_i negative hyperbolic}
      q_{i}(t)=\begin{pmatrix}
        x_{i}(t) \\
        y_{i}t \\
        \rho(t)S(\phi(t))\begin{pmatrix}
          \sinh{(\beta_{i})}\\
          \cosh{(\beta_{i})}
        \end{pmatrix}
      \end{pmatrix},\textrm{ }i\in\{1,...,n\},
    \end{align}
    and the $q_{1}$,...,$q_{n}$ lie on a polygon, then we call $q_{1}$,...,$q_{n}$ a \textit{negative hyperbolic polygonal rotopulsator}. If there exist twice differentiable functions $\theta$, $r\geq 0$, $r^{2}-\rho^{2}=-1$ and constants $\alpha_{1},...,\alpha_{n}\in[0,2\pi]$, such that \begin{align*}
      \begin{pmatrix}
        x_{i}(t)\\
        y_{i}(t)
      \end{pmatrix}=T(\theta(t))\begin{pmatrix}
        \sin{(\alpha_{i})}\\
          \cos{(\alpha_{i})}
      \end{pmatrix}
    \end{align*}
    then we call $q_{1}$,...,$q_{n}$ a \textit{negative elliptic-hyperbolic polygonal rotopulsator}.
  \end{definition}
  We will need the following lemma, which was proven in a more general setting in \cite{DK}, but as the proof for our particular case is not particularly long, we give a proof here as well:
  \begin{lemma}\label{Lemmaaaaa}
    If $q_{1}$,...,$q_{n}$ is a negative hyperbolic polygonal rotopulsator, or a negative elliptic-hyperbolic polygonal rotopulsator with functions $\rho$ and $\phi$ as in Definition~\ref{Definition negative hyperbolic}, then $2\rho'\phi'+\rho\phi''=0$.
  \end{lemma}
  \begin{proof}
   Using the wedge product, it was proven in \cite{D3} that
   \begin{align*}
     \sum\limits_{i=1}^{n}m_{i}q_{i}\wedge\ddot{q}_{i}=\mathbf{0},
   \end{align*}
   where $\mathbf{0}$ is the zero bivector.
   If $e_{1}$, $e_{2}$, $e_{3}$ and $e_{4}$ are the standard basis vectors in $\mathbb{R}^{4}$, then
   \begin{align}\label{Wedge q}
     0e_{3}\wedge e_{4}&=\sum\limits_{i=1}^{n}m_{i}(q_{i3}\ddot{q}_{i4}-q_{i4}\ddot{q}_{i3})e_{3}\wedge e_{4}.
   \end{align}
   As $q_{i3}=\rho\cosh{(\beta_{i}+\phi)}$ and $q_{i4}=\rho\sinh{(\beta_{i}+\phi)}$ by Definition~\ref{Definition negative hyperbolic}, we have that
   \begin{align}\label{Wedge q2}
     & q_{i3}\ddot{q}_{i4}-q_{i4}\ddot{q}_{i3}=\det{\begin{pmatrix}
       q_{i3} & \ddot{q}_{i3} \\
       q_{i4} & \ddot{q}_{i4}
     \end{pmatrix}}=\det{\begin{pmatrix}
       \rho\begin{pmatrix}
         \cosh{(\beta_{i}+\phi)} \\
         \sinh{(\beta_{i}+\phi)}
       \end{pmatrix} & \left(\rho\begin{pmatrix}
         \cosh{(\beta_{i}+\phi)} \\
         \sinh{(\beta_{i}+\phi)}
       \end{pmatrix}\right)''
     \end{pmatrix}}.
   \end{align}
   Because
   \begin{align*}
     \left(\rho\begin{pmatrix}
         \cosh{(\beta_{i}+\phi)} \\
         \sinh{(\beta_{i}+\phi)}
       \end{pmatrix}\right)''&=\rho''\begin{pmatrix}
         \cosh{(\beta_{i}+\phi)} \\
         \sinh{(\beta_{i}+\phi)}
       \end{pmatrix}+2\rho'\phi'\begin{pmatrix}
         \sinh{(\beta_{i}+\phi)} \\
         \cosh{(\beta_{i}+\phi)}
       \end{pmatrix}\\
       &+\rho\phi''\begin{pmatrix}
         \sinh{(\beta_{i}+\phi)} \\
         \cosh{(\beta_{i}+\phi)}
       \end{pmatrix}+\rho(\phi')^{2}\begin{pmatrix}
         \cosh{(\beta_{i}+\phi)} \\
         \sinh{(\beta_{i}+\phi)}
       \end{pmatrix},
   \end{align*}
   using that the determinant of a matrix with two identical rows is zero, we can rewrite (\ref{Wedge q2}) as
   \begin{align*}
     q_{i3}\ddot{q}_{i4}-q_{i4}\ddot{q}_{i3}&=0+\rho\left(2\rho'\phi'+\phi''\right)\det{\begin{pmatrix}
         \cosh{(\beta_{i}+\phi)} & \sinh{(\beta_{i}+\phi)} \\
         \sinh{(\beta_{i}+\phi)} & \cosh{(\beta_{i}+\phi)}
       \end{pmatrix}}\\
       &=\rho\left(2\rho'\phi'+\phi''\right)\cdot(-1).
   \end{align*}
   So combined with (\ref{Wedge q}), we get
   \begin{align*}
     0e_{3}\wedge e_{4}=-\rho\sum\limits_{i=1}^{n}m_{i}(2\rho'\phi'+\rho\phi'')e_{3}\wedge e_{4}=-\rho(2\rho'\phi'+\rho\phi'')\sum\limits_{i=1}^{n}m_{i}e_{3}\wedge e_{4},
   \end{align*}
   giving that indeed $2\rho'\phi'+\rho\phi''=0$.
  \end{proof}
  \section{Proof of Theorem~\ref{Main Theorem 1}}\label{Section proof of main theorem 1}
  Let $q_{1}$,...,$q_{n}$ be a negative hyperbolic polygonal rotopulsator as in Definition~\ref{Definition negative hyperbolic}.
     Let $I$ be the $2\times 2$ identity matrix. Then inserting (\ref{q_i negative hyperbolic}) into (\ref{EquationsOfMotion Curved}) and multiplying both sides of the resulting system of equations for the third and fourth coordinates of $q_{i}$ from the left with $S(\phi+\beta_{i})^{-1}$
    gives, as now $\sigma=-1$,
    \begin{align*}
      &\left(\rho''I+(2\rho'\phi'+\rho\phi'')\begin{pmatrix}
        0 & 1 \\
        1 & 0
      \end{pmatrix}+\rho(\phi')^{2}I-\right)\rho\begin{pmatrix}
        0 \\
        1
      \end{pmatrix}\nonumber\\
      &=\sum\limits_{j=1, j\neq i}^{n}\frac{m_{j}\rho\left(\begin{pmatrix}
        \sinh(\beta_{j}-\beta_{i}) \\
        \cosh(\beta_{j}-\beta_{i})
      \end{pmatrix}+(q_{i}\odot q_{j})\begin{pmatrix}
        0\\
        1
      \end{pmatrix}\right)}{((q_{i}\odot q_{j})^{2}-1)^{\frac{3}{2}}}\\
      &+((x_{i}')^{2}+(y_{i}')^{2}-(\rho')^{2}+\rho^{2}(\phi')^{2})\rho\begin{pmatrix}
        0 \\
        1
      \end{pmatrix},
    \end{align*}
    which can be rewritten as
     \begin{align}\label{Second two equations positive positive + positive negative}
      &\left(\rho''I+(2\rho'\phi'+\rho\phi'')\begin{pmatrix}
        0 & 1 \\
        1 & 0
      \end{pmatrix}+\rho(\phi')^{2}I-((x_{i}')^{2}+(y_{i}')^{2}-((\rho')^{2}+\rho^{2}(\phi')^{2}))I\right)\rho\begin{pmatrix}
        0 \\
        1
      \end{pmatrix}\nonumber\\
      &=\sum\limits_{j=1, j\neq i}^{n}\frac{m_{j}\rho\left(\begin{pmatrix}
        \sinh(\beta_{j}-\beta_{i}) \\
        \cosh(\beta_{j}-\beta_{i})
      \end{pmatrix}+(q_{i}\odot q_{j})\begin{pmatrix}
        0\\
        1
      \end{pmatrix}\right)}{((q_{i}\odot q_{j})^{2}-1)^{\frac{3}{2}}}.
    \end{align}
    Collecting terms for the first coordinate on both sides of (\ref{Second two equations positive positive + positive negative}) gives
    \begin{align*}
      2\rho'\phi'+\rho\phi''=\sum\limits_{j=1, j\neq i}^{n}\frac{m_{j}\rho\sinh(\beta_{j}-\beta_{i})}{((q_{i}\odot q_{j})^{2}-1)^{\frac{3}{2}}}.
    \end{align*}
    By Lemma~\ref{Lemmaaaaa}, $2\rho'\phi'+\rho\phi''=0$, so
    \begin{align}\label{Identity fourth coordinate}
      0=\sum\limits_{j=1, j\neq i}^{n}\frac{m_{j}\rho\sinh(\beta_{j}-\beta_{i})}{((q_{i}\odot q_{j})^{2}-1)^{\frac{3}{2}}}.
    \end{align}
    Now let $\beta_{1}=\min\{\beta_{j}|j\in\{1,...,n\}\}$. Then $\sinh(\beta_{j}-\beta_{1})\geq 0$ and $\sinh(\beta_{j}-\beta_{1})=0$ if and only if $\beta_{j}=\beta_{1}$, so as $(\sigma-\sigma(q_{i}\odot q_{j})^{2})^{\frac{3}{2}}>0$, for (\ref{Identity fourth coordinate}) to hold, all $\beta_{j}$ have to be equal to $\beta_{1}$. This means that $q_{i3}$ and $q_{i4}$ are independent of $i$. Therefore, as the $q_{i}$ are vertices of a polygon, we have that $q_{i1}$ and $q_{i2}$ are coordinates of vertices of a polygon that lie on a circle of radius $r=\sqrt{\rho^{2}-1}$, which means that by Theorem~1.1 of \cite{T2} that if $\rho$ and therefore $r$ is not constant, the $q_{i}$ are the vertices of a regular polygon. If $\rho$ is constant, then $q_{1}$,...,$q_{n}$ is a polygonal relative equilibrium solution and by Theorem~1.2 of \cite{T5} there exists at most one such solution. This completes the proof.

\end{document}